\documentclass{article}
\usepackage[all]{xypic}
\usepackage{tikz}
\usepackage{listings}
\usepackage{graphicx}            
\usepackage{amsthm,amsfonts,amsmath,amssymb}
\usepackage{array,blkarray,multirow}
\usepackage{xcolor}

\newtheorem{theorem}{Theorem}
\newtheorem{problem}[theorem]{Problem}
\newtheorem{lemma}[theorem]{Lemma}
\newtheorem{corollary}[theorem]{Corollary}
\newtheorem{proposition}[theorem]{Proposition}

\def\A{\mathbf{A}}
\def\x{\mathbf{x}}

\def\0{\mathbf{0}}
\newcommand{\RR}{\mathbb{R}}
\newcommand{\ZZ}{\mathbb{Z}}

\usepackage[normalem]{ulem}

\begin{document}
\lstset{language=Python}          

\title{Existence of regular nut graphs and the Fowler Construction.}
\author{John Baptist Gauci, Toma\v{z} Pisanski and Irene Sciriha}
\maketitle
\begin{abstract}
In this paper the problem of the existence of regular nut graphs is addressed. A generalization of Fowler's Construction which is a local enlargement applied to a vertex in a graph is introduced to generate nut graphs of higher order. Let $N(\rho)$ denote the set of integers $n$ such that there exists a regular nut graph of degree $\rho$ and order $n$. It is proven that $N(3) = \{12\} \cup \{2k : k \geq 9\}$ and that
$N(4) = \{8,10,12\} \cup \{n: n \geq 14\}$. The problem of determining $N(\rho)$ for $\rho > 4$ remains completely open.

\end{abstract}

\bigskip


\bigskip


\section{Introduction} \label{Intro}

All the graphs considered in this paper are simple (that is, without loops or multiple edges). The 0--1 \textit{adjacency matrix} $\A=\A(G)=(a_{ij})$ of a labelled graph $G$ on $n$ vertices is a real and symmetric $n\times n$ matrix such that $a_{ij}=1$ if there is an edge between the vertices $i$ and $j$, and $a_{ij}=0$ otherwise. A graph $G$ is \textit{singular} if zero is an eigenvalue of $\A(G)$, and the multiplicity of the eigenvalue zero in the spectrum of $G$ is the \textit{nullity} of $G$, denoted by $\eta=\eta(G)$. Thus $G$ is singular if and only if $\eta(G)>0$. The \textit{eigenvectors} $\x$ of $\A(G)$ associated with the eigenvalue $\lambda$ are the nonzero vectors determined by $\A\x=\lambda\x$.  The vectors in the nullspace $\ker({\A})$ of $\A$ are called \textit{kernel eigenvectors}.

In $G$, the fact that $\A\x = \0$ can be interpreted as an assignment of the entries of $\x$ to the vertices of $G$, that is, $\x:V(G) \rightarrow \RR$, such that the sum of the values assigned to the neighbours of $v$ sums up to 0 for each $v \in V(G)$, and at least one vertex $v \in V(G)$ is assigned a non-zero value $\x(v) \neq 0$. It is worth noting that for any integer matrix $\A$ having integer eigenvalue $\lambda$, there exists an associated eigenvector $\x$ with integer entries. Moreover, one may choose the entries of $\x$ to have no nontrivial common divisor so that $\x$ is determined up to the sign. In general, any eigenvector $\x$ corresponding to the 0 eigenvalue with  non--zero entries will be called \emph{admissible}.

A vertex of $G$ corresponding to a non--zero entry in some kernel eigenvector is a \textit{core vertex} of $G$. A \textit{core graph} is a singular graph each of whose vertices is a core vertex, whereas a \textit{nut graph} is a core graph of nullity one (introduced in \cite{GutmanSciriha-MaxSing}).  We remark that since in a core graph each vertex is a core vertex, then there is a kernel eigenvector with all entries being non--zero. Thus, there is a very natural interpretation of core graphs, usually referred to in the literature as the \emph{zero sum rule} \cite{GraoGutTrin1977, GutmanSciriha2001}. A graph $G$ is a core graph if there exists an assignment of values from $\ZZ \setminus \{0\}$ such that the sum of the values on the neighbourhood of any vertex $v$ adds up to 0. Using our terminology, we can state the following proposition.

\begin{proposition}
A graph $G$ is a core graph if and only if its adjacency matrix has an admissable eigenvector.
\end{proposition}

There are exactly three nut graphs on 7 vertices (shown in Figure \ref{Fig-Nut7}) and none of smaller order \cite{ScirihaGutman-NutExt}. The first two are planar; the third one is toroidal. There are 13 nut graphs on 8 vertices. One of them, the antiprism graph $A_4$, is shown in Figure \ref{Fig-Nut8}. It is quartic (or 4-regular). One of the nine smallest cubic nut graphs  is one of the two asymmetric cubic graphs on 12 vertices and is the well-known Frucht graph, shown in Figure \ref{Fig-Frucht}.

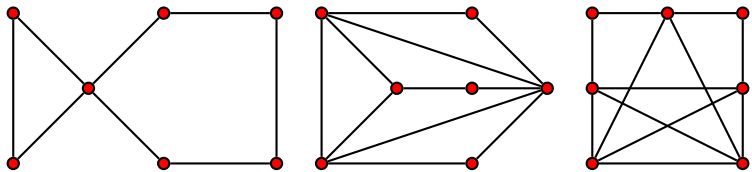
\begin{figure}[h]
\begin{center}
\begin{tabular}{ccc}
\begin{tikzpicture}[scale=1]
\tikzstyle{every path}=[draw, thick]
\tikzstyle{every node}=[draw, circle, fill=red, inner sep=1.5pt]
\node (v_1) at (0,1) {};
\node (v_2) at (1.5,1) {};
\node (v_3) at (1.5,-1) {};
\node (v_4) at (0,-1) {};
\node (v_5) at (-1,0) {};
\node (v_6) at (-2,-1) {};
\node (v_7) at (-2,1) {};
\path (v_1) -- (v_2);
\path (v_1) -- (v_5);
\path (v_2) -- (v_3);
\path (v_3) -- (v_4);
\path (v_4) -- (v_5);
\path (v_6) -- (v_5);
\path (v_6) -- (v_7);
\path (v_7) -- (v_5);
\end{tikzpicture}

&

\begin{tikzpicture}[scale=1]
\tikzstyle{every path}=[draw, thick]
\tikzstyle{every node}=[draw, circle, fill=red, inner sep=1.5pt]
\node (v_1) at (0,0) {};
\node (v_2) at (1,0) {};
\node (v_3) at (2,0) {};
\node (v_4) at (1,-1) {};
\node (v_5) at (-1,-1) {};
\node (v_6) at (-1,1) {};
\node (v_7) at (1,1) {};
\path (v_1) -- (v_2);
\path (v_1) -- (v_5);
\path (v_1) -- (v_6);
\path (v_2) -- (v_3);
\path (v_3) -- (v_4);
\path (v_3) -- (v_5);
\path (v_3) -- (v_6);
\path (v_3) -- (v_7);
\path (v_4) -- (v_5);
\path (v_5) -- (v_6);
\path (v_6) -- (v_7);
\end{tikzpicture}

&

\begin{tikzpicture}[scale=1]
\tikzstyle{every path}=[draw, thick]
\tikzstyle{every node}=[draw, circle, fill=red, inner sep=1.5pt]
\node (v_1) at (0,1) {};
\node (v_2) at (1,1) {};
\node (v_3) at (1,0) {};
\node (v_4) at (1,-1) {};
\node (v_5) at (-1,-1) {};
\node (v_6) at (-1,0) {};
\node (v_7) at (-1,1) {};
\path (v_1) -- (v_2);
\path (v_1) -- (v_4);
\path (v_1) -- (v_5);
\path (v_1) -- (v_7);
\path (v_2) -- (v_3);
\path (v_3) -- (v_4);
\path (v_3) -- (v_5);
\path (v_3) -- (v_6);
\path (v_4) -- (v_5);
\path (v_4) -- (v_6);
\path (v_5) -- (v_6);
\path (v_6) -- (v_7);
\end{tikzpicture}
\end{tabular}
\end{center}
\caption{The three nut graphs on 7 vertices} \label{Fig-Nut7}
\end{figure}

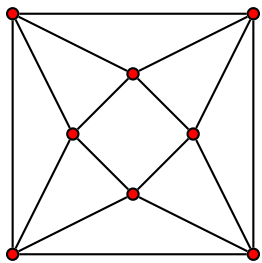
\begin{figure}[h]
\begin{center}
\begin{tikzpicture}[scale=0.8]
\tikzstyle{every path}=[draw, thick]
\tikzstyle{every node}=[draw, circle, fill=red, inner sep=1.5pt]
\node (v_1) at (1,0) {};
\node (v_2) at (0,-1) {};
\node (v_3) at (-1,0) {};
\node (v_4) at (0,1) {};
\node (v_5) at (2,2) {};
\node (v_6) at (2,-2) {};
\node (v_7) at (-2,-2) {};
\node (v_8) at (-2,2) {};
\path (v_1) -- (v_2);
\path (v_1) -- (v_4);
\path (v_1) -- (v_5);
\path (v_1) -- (v_6);
\path (v_2) -- (v_3);
\path (v_2) -- (v_6);
\path (v_2) -- (v_7);
\path (v_3) -- (v_4);
\path (v_3) -- (v_7);
\path (v_3) -- (v_8);
\path (v_4) -- (v_5);
\path (v_4) -- (v_8);
\path (v_5) -- (v_6);
\path (v_5) -- (v_8);
\path (v_6) -- (v_7);
\path (v_7) -- (v_8);
\end{tikzpicture}
\caption{The antiprism graph $A_4$ on 8 vertices is the smallest quartic nut graph.} \label{Fig-Nut8}
\end{center}
\end{figure}

\begin{figure}[h]
\begin{center}
\begin{tikzpicture}[scale=0.7]
\tikzstyle{every path}=[draw, thick]
\tikzstyle{every node}=[draw, circle, fill=red, inner sep=1.5pt]
\node (v_1) at (0,0) {};
\node (v_2) at (1.5,0) {};
\node (v_3) at (3,0) {};
\node (v_4) at (2,-3) {};
\node (v_5) at (-2,-3) {};
\node (v_6) at (-3,0) {};
\node (v_7) at (-2,2) {};
\node (v_8) at (0,3) {};
\node (v_9) at (2,2) {};
\node (v_10) at (0,1.5) {};
\node (v_11) at (-1.5,0) {};
\node (v_12) at (0,-1.5) {};
\path (v_1) -- (v_2);
\path (v_1) -- (v_10);
\path (v_1) -- (v_11);
\path (v_2) -- (v_3);
\path (v_2) -- (v_9);
\path (v_3) -- (v_4);
\path (v_3) -- (v_9);
\path (v_4) -- (v_5);
\path (v_4) -- (v_12);
\path (v_5) -- (v_6);
\path (v_5) -- (v_12);
\path (v_6) -- (v_7);
\path (v_6) -- (v_11);
\path (v_7) -- (v_8);
\path (v_7) -- (v_10);
\path (v_8) -- (v_9);
\path (v_8) -- (v_10);
\path (v_11) -- (v_12);
\end{tikzpicture}
\caption{The Frucht graph, one of the nine smallest cubic nut graphs on $12$ vertices.} \label{Fig-Frucht}
\end{center}
\end{figure}
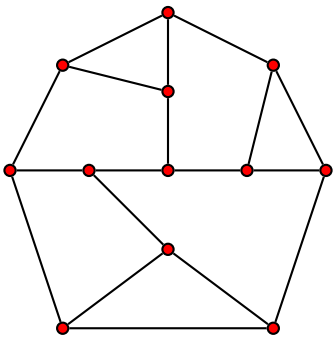

A recent computer search carried out by Coolsaet, Fowler and Goedgebeur  \cite{CoolFowlGoed-2017} shows that there exist no cubic nut graphs on 14 and 16 vertices. However, they discovered cubic nut graphs on 20, 22, 26 and 28 vertices.

The two principal results of this note can be stated as follows.

\begin{theorem} \label{CubicNut}
Cubic nut graphs on $n$ vertices exist if and only if $n$ is an even integer, $n \geq 12$ and $n \notin \{14,16\}$.
\end{theorem}

\begin{theorem} \label{QuarticNut}
Quartic nut graphs on $n$ vertices exist if and only if $n=8,10,12$ or $n\geq 14$.
\end{theorem}

These two results will be proved in Section \ref{ExistRegNuts}. In the proof of Theorem \ref{CubicNut} we use an ``construction'' so that starting from a graph $G$ on $n$ vertices with a vertex $v$ of degree 3 we can generate a graph $F(G,v)$ on $n+6$ vertices. Since this construction is due to Patrick Fowler  \cite{Irene-Coalesced,ScirihaFowler-Fullerenes}, we call it the {\it Fowler Construction}. Actually, in  Section \ref{FowlerExt}, we generalize it to vertices of arbitrary degree $\rho$ in such a way that the eigenvector entries corresponding to $V(G)\setminus \{v\}$ are preserved. In general, the graph is extended by $2\rho$ vertices. The main result of  Section \ref{FowlerExt} is that the nullity of $F(G,v)$ is the same as that of $G$. A direct consequence of this result that will enable us to construct an infinite family of nut graphs is that $F(G,v)$ preserves the property that $G$ is a nut graph.

In  Section \ref{ExistRegNuts}, we first prove the existence of cubic and quartic nut graphs. We proceed to determine the forbidden orders of cubic nut graphs. Using the Fowler Construction, we prove that there is an infinite family of quartic nut graphs. The construction of this family shows that quartic nut graphs  exist for all orders at least 14.


\section{The Fowler Construction} \label{FowlerExt}

Let $G$ be a graph, and $v$ any of its vertices of degree $\rho$, with neighbours labelled $u_1,u_2, \ldots, u_{\rho}$. We change $G$ to produce $F(G,v)$
as follows. Add $2\rho$ vertices $p_1, p_2, \ldots, p_{\rho}, q_1, q_2, \ldots, q_{\rho}$. For each $i\in\{1,\ldots,\rho\}$, remove all edges between $v$ and $u_i$ and insert edges between $v$ and $q_i$ and between each $p_i$ and $u_i$. Finally, add all edges from $p_i$ to $q_j$ for any pair $i,j\in\{1,\ldots,\rho\}$, $i \neq j$.  The resulting graph $F(G,v)$ is called \emph{the Fowler Construction} of $G$ at $v$. This process is depicted in Figure \ref{Fig-FowlerExt}.

\begin{figure}[h]
\begin{center}
\begin{tabular}{ >{\centering\arraybackslash} m{4.5cm} >{\centering\arraybackslash} m{1cm} >{\centering\arraybackslash} m{5.5cm} }
\begin{tikzpicture}[scale=0.9, every edge/.style = {draw, thick},
                    vertex/.style args = {#1 #2}{circle,
                                                draw, fill=red, inner sep=1.5pt,
                                                label=#1:#2}]
\path	node(v_1) [vertex=above $v$] at (0,0) {}
	node(u_1) [vertex=left $u_1$] at (-2,-1) {}
	node(u_2) [vertex=left $u_2$] at (-1,-1) {}
	node(u_rho) [vertex=right $u_\rho$] at (2,-1) {}
	(v_1) edge (u_1)
	(v_1) edge (u_2)
	(v_1) edge (u_rho);

 \path (u_1) -- node[auto=false]{\ldots} (u_rho);
\end{tikzpicture}

&

$\Rightarrow$

&

\begin{tikzpicture}[scale=0.9, every edge/.style = {draw, thick},
                    vertex/.style args = {#1 #2}{circle,
                                                draw, fill=red, inner sep=1.5pt,
                                                label=#1:#2}]
\path	node(v_1) [vertex=above $v$] at (0,0) {}
	node(q_1) [vertex=left $q_1$] at (-2,-1) {}
	node(q_2) [vertex=left $q_2$] at (-1,-1) {}
	node(q_rho) [vertex=right $q_\rho$] at (2,-1) {}
	node(p_1) [vertex=left $p_1$] at (-2,-2) {}
	node(p_2) [vertex=left $p_2$] at (-1,-2) {}
	node(p_rho) [vertex=right $p_\rho$] at (2,-2) {}
	node(u_1) [vertex=left $u_1$] at (-2,-3) {}
	node(u_2) [vertex=left $u_2$] at (-1,-3) {}
	node(u_rho) [vertex=right $u_\rho$] at (2,-3) {}
	(v_1) edge (q_1)
	(v_1) edge (q_2)
	(v_1) edge (q_rho)
	(q_1) edge (p_2)
	(q_1) edge (p_rho)
	(q_2) edge (p_1)
	(q_2) edge (p_rho)
	(q_rho) edge (p_1)
	(q_rho) edge (p_2)
	(p_1) edge (u_1)
	(p_2) edge (u_2)
	(p_rho) edge (u_rho);

 \path (q_1) -- node[auto=false]{\ldots} (q_rho);
 \path (p_1) -- node[auto=false]{\ldots} (p_rho);
 \path (u_1) -- node[auto=false]{\ldots} (u_rho);
\end{tikzpicture}

\\

\small{$G$} & & \small{$F(G,v)$}

\end{tabular}
\end{center}
\vspace{-15pt}\caption{The Fowler Construction.} \label{Fig-FowlerExt}
\end{figure}
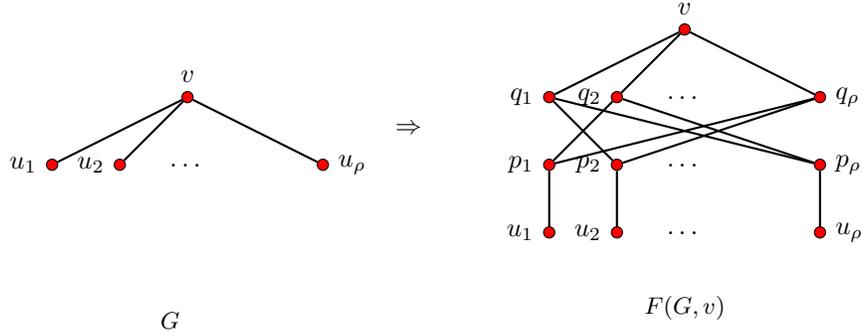

\begin{proposition}
If $G$ is a $\rho$--regular graph then any of its Fowler constructions $F(G,v)$ is also $\rho$-valent.
\end{proposition}

A simple but important lemma for the work that follows is given hereunder.

\begin{lemma} \label{EquivLabeling}
Let $G$ be a core graph, $\x$ an admissible eigenvector and $u$ and $v$ any two non-adjacent vertices.
Let $N(u,v)$ denote the set of vertices adjacent to both $u$ and $v$, $N(u-v)$ the vertices adjacent to $u$ but not to $v$,
and $N(v-u)$ the set of vertices adjacent to $v$ but not to $u$.  If $\{u'\} = N(u-v)$ and $\{v'\} = N(v-u)$ then
$\x(u') = \x(v')$.
\end{lemma}

\begin{proof}
Both  sums over the neighbours of $v$ and $u$ are 0, and the sum in the intersection is the same, hence the remainders have to be the same. If both remainders are singletons, their values must be equal.
\end{proof}

Now we are ready to prove the key theorem.

\begin{theorem}
A graph $G$ is a core graph if and only if its Fowler Construction $F(G,v)$ is a core graph, where $v$ is a vertex in $G$. Moreover, $G$ and $F(G,v)$ have equal nullities.
\end{theorem}

\begin{proof}
Let $u_1,\ldots, u_{\rho}$ be the neighbours of vertex $v$ in $G$. Assume first that $G$ is a core graph and that $\x$ is an admissible eigenvector. Let $\x(w)$ denote the entry of $\x$ at vertex $w$. Let $a = \x(v)$ and let $b_i = \x(u_i)$.  We now produce a vertex labelling $\x'$ of $F(G,v)$ as follows. Let $\x'(w) = \x(w)$ for any vertex $w \in V(G) \setminus \{v\}$. For all $i$, let $\x'(p_i) = a$, $\x'(q_i) = b_i$ and $\x'(v) = -(\rho-1)a$. If $\A(G)\x=\0$, then $\A(F(G,v))\x'=\0$. Thus, it follows that if $\x$ is a valid assignment of $G$ then $\x'$ is a valid assignment on $F(G,v)$.

Conversely, apply the above  Lemma \ref{EquivLabeling} to $F(G,v)$ and an admissible assignment $\x'$.  First consider vertices $q_i$ and $q_j$ and their neighbourhoods.  Lemma \ref{EquivLabeling} implies that $\x'(p_i) = \x'(p_j)$. Hence $\x'$ is constant on $p_i$, say $\x'(p_i) = t$.  Thus, it follows that $\x'(v) = -(\rho-1)t$. The second application of the lemma goes to vertices $v$ and $p_i$. It implies that for each $i$ the values $\x'(q_i)$ and $\x'(u_i)$ are equal, namely $\x'(q_i) = \x'(u_i)$.  Finally, let $\x(w) = \x'(w)$ for every $w \in V(G) \setminus \{v\}$ and let $\x(v) = t$. Hence, the existence of an admissible $\x'$ on $F(G,v)$ implies the existence of an admissible $\x$ on $G$.

To show that the nullities of $G$ and $F(G,v)$ are equal, let $N_G(v)=\{u_1,\ldots,u_{\rho}\}$ and $V(G)\setminus N[v]=\{w_1,\ldots,w_{n-\rho-1}\}$, where $n=|V(G)|$. The adjacency matrix $\A(G)$ can be partitioned into block matrices as follows:

$$
  \A(G) = \begin{blockarray}{cccccccc}
       v & w_1 & \ldots & w_{n-\rho-1} & v_1 & \ldots & v_{\rho} & \\
   \begin{block}{(c|ccc|ccc)c}
     0 & 0 & \ldots & 0 & 1 & \ldots & 1 & v\\ \cline{1-7}
     0 & \BAmulticolumn{3}{c|}{\multirow{3}{*}{$\bf B$}} & \BAmulticolumn{3}{c}{\multirow{3}{*}{$\bf C$}} &  w_1 \\
    \vdots  &  &  &  &  &  &  & \vdots \\
    0  &  &  &  &  &  &  & w_{n-\rho-1}\\ \cline{1-7}
   1 & \BAmulticolumn{3}{c|}{\multirow{3}{*}{${\bf C}^T$}} & \BAmulticolumn{3}{c}{\multirow{3}{*}{$\bf D$}} &  v_1  \\
   \vdots  &  &  &  &  &  &  & \vdots \\
   1  &  &  &  &  &  &  & v_{\rho}\\
		\end{block}
  \end{blockarray}
$$
where the submatrices $\bf B$, $\bf C$ and $\bf D$ encode the adjacencies between the respective vertex-sets.

The adjacency matrix of $\A(F(G,v))$ can similarly be partitioned as follows:
$$
\begin{blockarray}{cccccccccccccc}
       v & w_1 & \ldots & w_{n-\rho-1} & v_1 & \ldots & v_{\rho} & q_1 & \ldots & q_{\rho} & p_1 & \ldots & p_{\rho}\\
   \begin{block}{(c|ccc|ccc|ccc|ccc)c}
     0 & 0 & \ldots & 0 & 0 & \ldots & 0 & 1 & \ldots & 1 & 0 & \ldots & 0 & v\\ \cline{1-13}
     0 & \BAmulticolumn{3}{c|}{\multirow{3}{*}{$\bf B$}} & \BAmulticolumn{3}{c|}{\multirow{3}{*}{$\bf C$}} & \BAmulticolumn{3}{c|}{\multirow{3}{*}{$\bf 0$}} & \BAmulticolumn{3}{c}{\multirow{3}{*}{$\bf 0$}} &  w_1 \\
    \vdots  &  &  &  &  &  &  & & & & & & &\vdots \\
    0  &  &  &  &  &  &  & & & & & & &  w_{n-\rho-1}\\ \cline{1-13}
   0 & \BAmulticolumn{3}{c|}{\multirow{3}{*}{${\bf C}^T$}} & \BAmulticolumn{3}{c|}{\multirow{3}{*}{$\bf D$}} & \BAmulticolumn{3}{c|}{\multirow{3}{*}{$\bf 0$}} & \BAmulticolumn{3}{c}{\multirow{3}{*}{$\bf I$}} &  v_1  \\
   \vdots  &  &  &  &  &  &  & & & & & & & \vdots \\
   0  &  &  &  &  &  &  & & & & & & & v_{\rho}\\  \cline{1-13}
     1 & \BAmulticolumn{3}{c|}{\multirow{3}{*}{$\bf 0$}} & \BAmulticolumn{3}{c|}{\multirow{3}{*}{$\bf 0$}} & \BAmulticolumn{3}{c|}{\multirow{3}{*}{$\bf 0$}} & \BAmulticolumn{3}{c}{\multirow{3}{*}{$\bf J-I$}} &  q_1 \\
    \vdots  &  &  &  &  &  &  & & & & & & &\vdots \\
    1  &  &  &  &  &  &  & & & & & & &  q_{\rho}\\ \cline{1-13}
     0 & \BAmulticolumn{3}{c|}{\multirow{3}{*}{$\bf 0$}} & \BAmulticolumn{3}{c|}{\multirow{3}{*}{$\bf I$}} & \BAmulticolumn{3}{c|}{\multirow{3}{*}{$\bf J-I$}} & \BAmulticolumn{3}{c}{\multirow{3}{*}{$\bf 0$}} &  p_1 \\
    \vdots  &  &  &  &  &  &  & & & & & & &\vdots \\
    0  &  &  &  &  &  &  & & & & & & &  p_{\rho}\\
		\end{block}
  \end{blockarray}
$$
where $\bf I$ is the identity matrix and $\bf J$ is the all-one matrix.

Elementary row and corresponding column operations which leave the rank unchanged are performed by replacing the rows and columns corresponding to $v_1,\ldots, v_{\rho}$ by $v_1+q_1,\ldots, v_{\rho}+q_{\rho}$, respectively, to obtain the matrix

$$
\begin{blockarray}{cccccccccccccc}
       v & w_1 & \ldots & w_{n-\rho-1} & v_1 & \ldots & v_{\rho} & q_1 & \ldots & q_{\rho} & p_1 & \ldots & p_{\rho}\\
   \begin{block}{(c|ccc|ccc|ccc|ccc)c}
     0 & 0 & \ldots & 0 & 1 & \ldots & 1 & 1 & \ldots & 1 & 0 & \ldots & 0 & v\\ \cline{1-13}
     0 & \BAmulticolumn{3}{c|}{\multirow{3}{*}{$\bf B$}} & \BAmulticolumn{3}{c|}{\multirow{3}{*}{$\bf C$}} & \BAmulticolumn{3}{c|}{\multirow{3}{*}{$\bf 0$}} & \BAmulticolumn{3}{c}{\multirow{3}{*}{$\bf 0$}} &  w_1 \\
    \vdots  &  &  &  &  &  &  & & & & & & &\vdots \\
    0  &  &  &  &  &  &  & & & & & & &  w_{n-\rho-1}\\ \cline{1-13}
    1 & \BAmulticolumn{3}{c|}{\multirow{3}{*}{${\bf C}^T$}} & \BAmulticolumn{3}{c|}{\multirow{3}{*}{$\bf D$}} & \BAmulticolumn{3}{c|}{\multirow{3}{*}{$\bf 0$}} & \BAmulticolumn{3}{c}{\multirow{3}{*}{$\bf J$}} &  v_1  \\
   \vdots  &  &  &  &  &  &  & & & & & & & \vdots \\
   1  &  &  &  &  &  &  & & & & & & & v_{\rho}\\  \cline{1-13}
     1 & \BAmulticolumn{3}{c|}{\multirow{3}{*}{$\bf 0$}} & \BAmulticolumn{3}{c|}{\multirow{3}{*}{$\bf 0$}} & \BAmulticolumn{3}{c|}{\multirow{3}{*}{$\bf 0$}} & \BAmulticolumn{3}{c}{\multirow{3}{*}{$\bf J-I$}} &  q_1 \\
    \vdots  &  &  &  &  &  &  & & & & & & &\vdots \\
    1  &  &  &  &  &  &  & & & & & & &  q_{\rho}\\ \cline{1-13}
     0 & \BAmulticolumn{3}{c|}{\multirow{3}{*}{$\bf 0$}} & \BAmulticolumn{3}{c|}{\multirow{3}{*}{$\bf J$}} & \BAmulticolumn{3}{c|}{\multirow{3}{*}{$\bf J-I$}} & \BAmulticolumn{3}{c}{\multirow{3}{*}{$\bf 0$}} &  p_1 \\
    \vdots  &  &  &  &  &  &  & & & & & & &\vdots \\
    0  &  &  &  &  &  &  & & & & & & &  p_{\rho}\\
		\end{block}
  \end{blockarray}
$$

The rank of $\bf J-I$ is full. Hence,  ${\rm rk}(\A(F(G,v))) \geq {\rm rk}(\A(G))+2\rho=n-1+2\rho$, since the last $2\rho$ rows/columns are linearly independent of all the other rows/columns. Thus, $\eta(F(G,v))\leq \eta(G)$, and by the first part of the proof, we get $\eta(F(G,v))=\eta(G)$.
\end{proof}

The result we need most is a straightforward corollary of the above theorem.

\begin{theorem}
$G$ is a nut graph if and only if its Fowler Construction $F(G,v)$ is a nut graph.
\end{theorem}


\section{The existence problem for regular nut graphs} \label{ExistRegNuts}

The machinery that we prepared in the previous section will be used in the following result. In the construction of $F(G,v)$, $2\rho$ vertices are added to $G$, where $\rho$ is the degree of vertex $v$ in $G$. Each of the new vertices in $F(G,v)$ acquire the degree $\rho$, while all the other vertices retain the degree in $G$.

\begin{corollary} Let $G$ be a nut graph on $n$ vertices and let $v$ be any of its vertices.  If the degree of $v$ is $\rho$, then there exists a nut graph $G'$ on $n + 2\rho$ vertices. Moreover, if $G$ is regular, then $G'$ is regular.
\end{corollary}

In the sequel, we establish the existence, or otherwise, of regular nut graphs. We shall require the following result.

\begin{lemma} \cite{ScirihaGutman-NutExt} \label{4-foldSubdvn}
Let $G$ be a graph and $e$ any of its edges. Let $G'$ be a graph obtained from $G$ by a 4-fold subdivision of $e$. Then $G$ is a core graph if and
only if $G'$ is a core graph; and $G$ is a nut graph if and only if $G'$ is a nut graph.
\end{lemma}

The case when $\rho=2$ is settled in the following result.

\begin{theorem}
  There are no regular nut graphs of degree two.
\end{theorem}

\begin{proof}
It is easy to check that the cycles $C_3,C_5$ and $C_6$ are non-singular graphs and that $C_4$ is a core graph that is not a nut graph. By Lemma \ref{4-foldSubdvn} and induction, it readily follows that a cycle $C_n$ on $n$ vertices is singular if and only if $n$ is divisible by $4$ and also that any cycle which is singular is a core graph but not a nut graph.
\end{proof}

The case when $\rho=3$ is presented in  Theorem \ref{CubicNut} stated in the introduction, which we can now proceed to prove.

\begin{proof} (Proof of  Theorem \ref{CubicNut}.)
For cubic graphs, $n$ must be even. Non-existence for $n < 12$ and $n \in \{14,16\}$ has been shown by computer search. Since a cubic nut graph exists for $n = 12$, then cubic nut graphs for $n = 12+6 = 18$  and $n = 18+6 = 24$ exist as well.  Using the same argument, the existence of cubic nut graphs of order $n$, for $n$ divisible by $6$ is guaranteed. Using the cubic nut graphs for $n=20$ and for $n=22$ as seeds (shown in Figure \ref{Fig-Nut2628}), the remaining cases are also covered.
\end{proof}

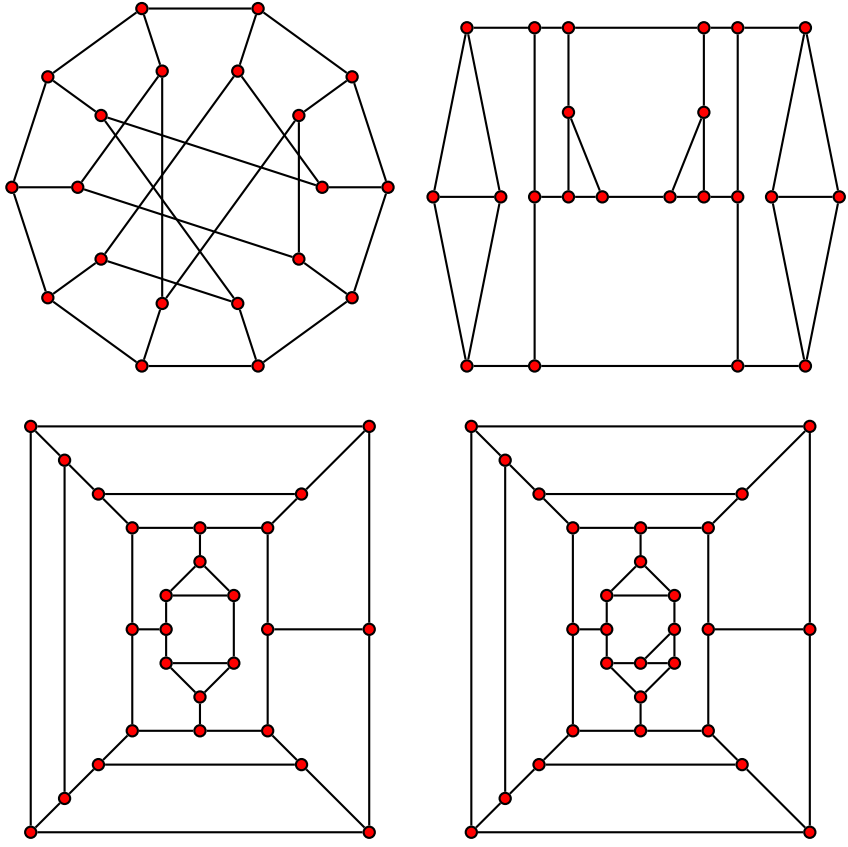
\begin{figure}[h]
\begin{center}
\begin{tabular}{cc}
\begin{tikzpicture}[scale=2.5]
\tikzstyle{every path}=[draw, thick]
\tikzstyle{every node}=[draw, circle, fill=red, inner sep=1.5pt]
\node (v_0) at (1.0,0.0) {};
\node (v_1) at (0.809017, 0.587785) {};
\node (v_2) at (0.309017, 0.951056) {};
\node (v_3) at (-0.309017, 0.951056) {};
\node (v_4) at (-0.809017, 0.587785) {};
\node (v_5) at (-1.0, 0.0) {};
\node (v_6) at (-0.809017, -0.587785) {};
\node (v_7) at (-0.309017, -0.951057) {};
\node (v_8) at (0.309017, -0.951057) {};
\node (v_9) at (0.809017, -0.587785) {};
\node (v_10) at (0.65, 0.0) {};
\node (v_11) at (0.525861, 0.382060) {};
\node (v_12) at (0.200861, 0.618187) {};
\node (v_13) at (-0.200861, 0.618187) {};
\node (v_14) at (-0.525861, 0.382060) {};
\node (v_15) at (-0.65, 0.0) {};
\node (v_16) at (-0.525851, -0.382060) {};
\node (v_17) at (-0.200861, -0.618187) {};
\node (v_18) at (0.200861, -0.6181875) {};
\node (v_19) at (0.525851, -0.382060) {};
\path (v_0) -- (v_1);
\path (v_0) -- (v_9);
\path (v_0) -- (v_10);
\path (v_1) -- (v_2);
\path (v_1) -- (v_11);
\path (v_2) -- (v_3);
\path (v_2) -- (v_12);
\path (v_3) -- (v_4);
\path (v_3) -- (v_13);
\path (v_4) -- (v_5);
\path (v_4) -- (v_14);
\path (v_5) -- (v_6);
\path (v_5) -- (v_15);
\path (v_6) -- (v_7);
\path (v_6) -- (v_16);
\path (v_7) -- (v_8);
\path (v_7) -- (v_17);
\path (v_8) -- (v_9);
\path (v_8) -- (v_18);
\path (v_9) -- (v_19);
\path (v_10) -- (v_12);
\path (v_10) -- (v_14);
\path (v_11) -- (v_17);
\path (v_11) -- (v_19);
\path (v_12) -- (v_16);
\path (v_13) -- (v_15);
\path (v_13) -- (v_17);
\path (v_14) -- (v_18);
\path (v_15) -- (v_19);
\path (v_16) -- (v_18);
\end{tikzpicture}

&

\begin{tikzpicture}[scale=0.45]
\tikzstyle{every path}=[draw, thick]
\tikzstyle{every node}=[draw, circle, fill=red, inner sep=1.5pt]
\node (v_0) at (-5.00000, -5.00000) {};
\node (v_1) at (-3.00000, -5.00000) {};
\node (v_2) at (3.00000, -5.00000) {};
\node (v_3) at (5.00000, -5.00000) {};
\node (v_4) at (-6.00000, 0.00000) {};
\node (v_5) at (-4.00000, 0.00000) {};
\node (v_6) at (-3.00000, 0.00000) {};
\node (v_7) at (-2.00000, 0.00000) {};
\node (v_8) at (-1.00000, 0.00000) {};
\node (v_9) at (1.00000, 0.00000) {};
\node (v_10) at (2.00000, 0.00000) {};
\node (v_11) at (3.00000, 0.00000) {};
\node (v_12) at (4.00000, 0.00000) {};
\node (v_13) at (6.00000, 0.00000) {};
\node (v_14) at (-2.00000, 2.500000) {};
\node (v_15) at (2.00000, 2.50000) {};
\node (v_16) at (-5.00000, 5.00000) {};
\node (v_17) at (-3.00000, 5.00000) {};
\node (v_18) at (-2.00000, 5.00000) {};
\node (v_19) at (2.00000, 5.00000) {};
\node (v_20) at (3.00000, 5.00000) {};
\node (v_21) at (5.00000, 5.00000) {};
\path (v_0) -- (v_1);
\path (v_0) -- (v_4);
\path (v_0) -- (v_5);
\path (v_1) -- (v_2);
\path (v_1) -- (v_6);
\path (v_2) -- (v_3);
\path (v_2) -- (v_11);
\path (v_3) -- (v_12);
\path (v_3) -- (v_13);
\path (v_4) -- (v_5);
\path (v_4) -- (v_16);
\path (v_5) -- (v_16);
\path (v_6) -- (v_7);
\path (v_6) -- (v_17);
\path (v_7) -- (v_8);
\path (v_7) -- (v_14);
\path (v_8) -- (v_9);
\path (v_8) -- (v_14);
\path (v_9) -- (v_10);
\path (v_9) -- (v_15);
\path (v_10) -- (v_11);
\path (v_10) -- (v_15);
\path (v_11) -- (v_20);
\path (v_12) -- (v_13);
\path (v_12) -- (v_21);
\path (v_13) -- (v_21);
\path (v_14) -- (v_18);
\path (v_15) -- (v_19);
\path (v_16) -- (v_17);
\path (v_17) -- (v_18);
\path (v_18) -- (v_19);
\path (v_19) -- (v_20);
\path (v_20) -- (v_21);
\end{tikzpicture}

\vspace{0.5cm}

\\

\begin{tikzpicture}[scale=0.9]
\tikzstyle{every path}=[draw, thick]
\tikzstyle{every node}=[draw, circle, fill=red, inner sep=1.5pt]
\node (v_0) at (-2,2.5) {};
\node (v_1) at (-1.5,2) {};
\node (v_2) at (-2.5,3) {};
\node (v_3) at (-2,-2.5) {};
\node (v_4) at (-1,1.5) {};
\node (v_5) at (1.5,2) {};
\node (v_6) at (2.5,3) {};
\node (v_7) at (-2.5,-3) {};
\node (v_8) at (-1.5,-2) {};
\node (v_9) at (-1,0) {};
\node (v_10) at (0,1.5) {};
\node (v_11) at (1,1.5) {};
\node (v_12) at (2.5,0) {};
\node (v_13) at (2.5,-3) {};
\node (v_14) at (1.5,-2) {};
\node (v_15) at (-1,-1.5) {};
\node (v_16) at (-0.5,0) {};
\node (v_17) at (0,1) {};
\node (v_18) at (1,0) {};
\node (v_19) at (1,-1.5) {};
\node (v_20) at (0,-1.5) {};
\node (v_21) at (-0.5,-0.5) {};
\node (v_22) at (-0.5,0.5) {};
\node (v_23) at (0.5,0.5) {};
\node (v_24) at (0,-1) {};
\node (v_25) at (0.5,-0.5) {};
\path (v_0) -- (v_1);
\path (v_0) -- (v_2);
\path (v_0) -- (v_3);
\path (v_1) -- (v_4);
\path (v_1) -- (v_5);
\path (v_2) -- (v_6);
\path (v_2) -- (v_7);
\path (v_3) -- (v_7);
\path (v_3) -- (v_8);
\path (v_4) -- (v_9);
\path (v_4) -- (v_10);
\path (v_5) -- (v_6);
\path (v_5) -- (v_11);
\path (v_6) -- (v_12);
\path (v_7) -- (v_13);
\path (v_8) -- (v_14);
\path (v_8) -- (v_15);
\path (v_9) -- (v_15);
\path (v_9) -- (v_16);
\path (v_10) -- (v_11);
\path (v_10) -- (v_17);
\path (v_11) -- (v_18);
\path (v_12) -- (v_13);
\path (v_12) -- (v_18);
\path (v_13) -- (v_14);
\path (v_14) -- (v_19);
\path (v_15) -- (v_20);
\path (v_16) -- (v_21);
\path (v_16) -- (v_22);
\path (v_17) -- (v_22);
\path (v_17) -- (v_23);
\path (v_18) -- (v_19);
\path (v_19) -- (v_20);
\path (v_20) -- (v_24);
\path (v_21) -- (v_24);
\path (v_21) -- (v_25);
\path (v_22) -- (v_23);
\path (v_23) -- (v_25);
\path (v_24) -- (v_25);
\end{tikzpicture}

&

\begin{tikzpicture}[scale=0.9]
\tikzstyle{every path}=[draw, thick]
\tikzstyle{every node}=[draw, circle, fill=red, inner sep=1.5pt]
\node (v_0) at (-2,2.5) {};
\node (v_1) at (-1.5,2) {};
\node (v_2) at (-2.5,3) {};
\node (v_3) at (-2,-2.5) {};
\node (v_4) at (-1,1.5) {};
\node (v_5) at (1.5,2) {};
\node (v_6) at (2.5,3) {};
\node (v_7) at (-2.5,-3) {};
\node (v_8) at (-1.5,-2) {};
\node (v_9) at (-1,0) {};
\node (v_10) at (0,1.5) {};
\node (v_11) at (1,1.5) {};
\node (v_12) at (2.5,0) {};
\node (v_13) at (2.5,-3) {};
\node (v_14) at (1.5,-2) {};
\node (v_15) at (-1,-1.5) {};
\node (v_16) at (-0.5,0) {};
\node (v_17) at (0,1) {};
\node (v_18) at (1,0) {};
\node (v_19) at (1,-1.5) {};
\node (v_20) at (0,-1.5) {};
\node (v_21) at (-0.5,-0.5) {};
\node (v_22) at (-0.5,0.5) {};
\node (v_23) at (0.5,0.5) {};
\node (v_24) at (0,-1) {};
\node (v_25) at (0.5,-0.5) {};
\node (v_26) at (0.5,0) {};
\node (v_27) at (0,-0.5) {};
\path (v_0) -- (v_1);
\path (v_0) -- (v_2);
\path (v_0) -- (v_3);
\path (v_1) -- (v_4);
\path (v_1) -- (v_5);
\path (v_2) -- (v_6);
\path (v_2) -- (v_7);
\path (v_3) -- (v_7);
\path (v_3) -- (v_8);
\path (v_4) -- (v_9);
\path (v_4) -- (v_10);
\path (v_5) -- (v_6);
\path (v_5) -- (v_11);
\path (v_6) -- (v_12);
\path (v_7) -- (v_13);
\path (v_8) -- (v_14);
\path (v_8) -- (v_15);
\path (v_9) -- (v_15);
\path (v_9) -- (v_16);
\path (v_10) -- (v_11);
\path (v_10) -- (v_17);
\path (v_11) -- (v_18);
\path (v_12) -- (v_13);
\path (v_12) -- (v_18);
\path (v_13) -- (v_14);
\path (v_14) -- (v_19);
\path (v_15) -- (v_20);
\path (v_16) -- (v_21);
\path (v_16) -- (v_22);
\path (v_17) -- (v_22);
\path (v_17) -- (v_23);
\path (v_18) -- (v_19);
\path (v_19) -- (v_20);
\path (v_20) -- (v_24);
\path (v_21) -- (v_24);
\path (v_21) -- (v_27);
\path (v_22) -- (v_23);
\path (v_23) -- (v_26);
\path (v_24) -- (v_25);
\path (v_25) -- (v_26);
\path (v_25) -- (v_27);
\path (v_26) -- (v_27);
\end{tikzpicture}
\end{tabular}
\end{center}
\vspace{-14pt}\caption{Cubic nut graphs of order 20, 22, 26 and 28.} \label{Fig-Nut2628}
\end{figure}

The idea used in the previous proof whereby repeated Fowler constructions are used starting from a seed graph implies the next result.

\begin{proposition}
If there exists a $\rho$-regular nut graph, then there exist infinitely many such graphs.
\end{proposition}

We now shift our attention to the case $\rho=4$. However, before proceeding further, we recall that a {\bf circulant matrix} is an $N\times N$ matrix ${\bf C}=(a_{k,j}: k,j=1,2,\ldots,N)$ where $a_{k,j}=a_{(j-k)\mod N}$, denoted by ${\bf C}=\langle a_0,a_1,\ldots,a_{N-1}\rangle$. The eigenvalues are given by $\lambda_r=\sum_{j=0}^{N-1} a_j \omega^{jr}$, where $r\in\{0,1,\ldots,N-1\}$ and $\omega= \exp{\frac{2\pi i}{N}}$.

\begin{proposition} \label{QuarticAntiPrism}
The quartic antiprism graph $A_n$ of order $2n$ is a nut graph if and only if $n$ is not divisible by 3.
\end{proposition}

\begin{proof}
The antiprism graph $A_n$ is isomorphic to the circulant graph ${\rm Ci}_{2n}(1,2)$ described by the circulant matrix
$${\bf C}=\langle
\begin{matrix}
0 & 1 & 1 & \0 & 0 & 1 & 1
\end{matrix}
\rangle^T,$$
where $\0$ is the $(2n-6)$--row-vector composed of zero entries.

Thus, for $r\in\{0,1,\ldots,2n-1\}$, $\lambda_r=\omega^r+\omega^{2r}+\omega^{(2n-2)r}+\omega^{(2n-1)r}=\left(\omega^r+\omega^{-r}\right)+\left(\omega^{2r}+\omega^{-2r}\right)=2\left(2\cos{\frac{\pi r}{n}}-1\right)\left(\cos{\frac{\pi r}{n}}+1\right)$. Hence, $\lambda_r=0$ if and only if $\cos{\frac{\pi r}{n}}=\frac{1}{2}$ or $\cos{\frac{\pi r}{n}}=-1$, that is, if and only if $r=\pm \frac{n}{3}+2kn$ or $r=n+2kn$, where $k\in\mathbb{Z}$. Now, if $n$ is divisible by 3, then $\lambda_r=0$ if and only if $r\in\{\frac{n}{3},n,\frac{5n}{3}\}$, whereas if $n$ is not divisible by 3, then $\lambda_r=0$ if and only if $r=n$. Hence in the former case, the nullity of $A_n$ is three, whereas in the latter case, the nullity of $A_n$ is one.

Starting with the case when $n$ is divisible by 3, the vector ${\bf v}_1$ composed of $\frac{n}{3}$ copies of the vector $\begin{pmatrix}0 & 0 & -1 & 0 & 0 &1\end{pmatrix}^T$ is a kernel eigenvector of ${\bf C}$ because in ${\bf C} {\bf v}_1$, each row of ${\bf C}$ captures either exactly two zero entries of ${\bf v}_1$ or the pair $\{-1,1\}$, and thus ${\bf C} {\bf v}_1=\0$. A similar argument can be used on the vector ${\bf v}_2$ composed of $\frac{n}{3}$ copies of the vector $\begin{pmatrix}0 & -1 & 0 & 0 & 1 &0\end{pmatrix}^T$ and on the vector ${\bf v}_3$ composed of $\frac{n}{3}$ copies of the vector $\begin{pmatrix}-1 & 0 & 0 & 1 & 0 &0\end{pmatrix}^T$. Hence ${\bf v}_1$, ${\bf v}_2$ and ${\bf v}_3$ are three linearly independent kernel eigenvectors of ${\bf C}_{2n}$, implying that $A_n$ is core graph but not a nut graph.

When $n$ is not divisible by 3, it can be readily checked that the vector ${\bf v}_4$ composed by taking $n$ copies of the vector $\begin{pmatrix}-1 & 1\end{pmatrix}^T$ is a kernel eigenvector of ${\bf C}$ and thus $A_n$ is a nut graph.
\end{proof}

We remark that a general nullspace vector ${\bf x}$ of the adjacency matrix $\bf A$ of the antiprism $A_n$ can be obtained by considering ${\bf Ax}={\bf 0}$. We adopt a labelling of $A_n$ similar to that shown in  Figure \ref{Fig-AntiprismA9} for the case $n=9$. For any value of $n$, the vertices $1,2,3$ and $4$ are given nullspace vector weights $a,b,c$ and $d$, respectively. Then
\begin{align*}
{\bf x} = \left( \right. &a \hspace{15pt} b \hspace{15pt} c \hspace{15pt} d \hspace{15pt} (-a-b-d) \hspace{10pt} (a-c+d) \hspace{10pt} (-a-2d)\\
        &  (2a+b+2d)\hspace{10pt} (-2a+c-2d) \hspace{10pt} (2a+3d) \hspace{10pt} (-3a-b-3d)\\
        & (3a-c+3d) \hspace{10pt} (-3a-4d)\hspace{10pt} (4a+b+4d) \hspace{10pt} (-4a+c-4d)\\
        & \left. (4a+5d)\hspace{10pt} (-5a-b-5d) \hspace{10pt} (5a-c+5d) \hspace{10pt} \ldots ~\right)^T.
\end{align*}

\begin{figure}
\begin{center}
\begin{tikzpicture}[scale=2.7, every edge/.style = {draw, thick},
                    vertex/.style args = {#1 #2}{circle,
                                                draw, fill=red, inner sep=1.5pt,
                                                label=#1:#2}]
\path	node(v_1) [vertex=above 1] at (0,1) {}
	node(v_3) [vertex=left 3] at (-0.642788,0.766044) {}
	node(v_5) [vertex=left 5] at (-0.984808,0.173648) {}
	node(v_7) [vertex=left 7] at (-0.866025,-0.5) {}
	node(v_9) [vertex=below 9] at (-0.34202,-0.939693) {}
	node(v_11) [vertex=below 11] at (0.34202,-0.939693) {}
	node(v_13)  [vertex=right 13] at (0.866025,-0.5) {}
	node(v_15)  [vertex=right 15] at (0.984808,0.173648) {}
	node(v_17)  [vertex=right 17] at (0.642788,0.766044) {}
	node(v_2)  [vertex=below 2] at (-0.17101,0.469846) {}
	node(v_4)  [vertex=right 4] at (-0.433013,0.25) {}
	node(v_6)  [vertex=right 6] at (-0.492404,-0.0868241) {}
	node(v_8)  [vertex=right 8] at (-0.321394,-0.383022) {}
	node(v_10)  [vertex=below 10] at (0.,-0.5) {}
	node(v_12)  [vertex=left 12] at (0.321394,-0.383022) {}
	node(v_14)  [vertex=left 14] at (0.492404,-0.0868241) {}
	node(v_16)  [vertex=left 16] at (0.433013,0.25) {}
	node(v_18) [vertex=below 18]  at (0.17101,0.469846) {}
	(v_1) edge (v_2)
	(v_1) edge (v_3)
	(v_2) edge (v_3)
	(v_2) edge (v_4)
	(v_3) edge (v_4)	
	(v_3) edge (v_5)
	(v_4) edge (v_5)
	(v_4) edge (v_6)
	(v_5) edge (v_6)
	(v_5) edge (v_7)
	(v_6) edge (v_7)
	(v_6) edge (v_8)
	(v_7) edge (v_8)
	(v_7) edge (v_9)
	(v_8) edge (v_9)
	(v_8) edge (v_10)
	(v_9) edge (v_10)
	(v_9) edge (v_11)
	(v_10) edge (v_11)
	(v_10) edge (v_12)
	(v_11) edge (v_12)
	(v_11) edge (v_13)
	(v_12) edge (v_13)
	(v_12) edge (v_14)
	(v_13) edge (v_14)
	(v_13) edge (v_15)
	(v_14) edge (v_15)
	(v_14) edge (v_16)
	(v_15) edge (v_16)
	(v_15) edge (v_17)
	(v_16) edge (v_17)
	(v_16) edge (v_18)
	(v_17) edge (v_18)
	(v_17) edge (v_1)
	(v_18) edge (v_1)
	(v_18) edge (v_2);
\end{tikzpicture}
\end{center}
\vspace{-14pt} \caption{The antiprism $A_9$.} \label{Fig-AntiprismA9}
\end{figure}
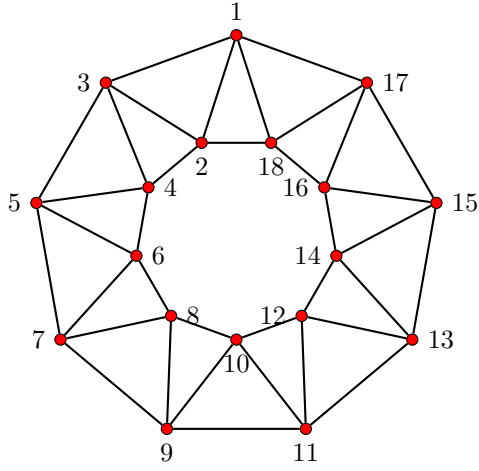

For $n=3k,$ $k\in \mathbb{Z}^+$, ${\bf Ax}={\bf 0}$  yields
$${\bf Q}_{3k}\left(\begin{array}{c}
a\\ b\\ c\\ d
\end{array}\right)
=\left( \begin{array}{cccc}
  -2k&0&0&-2k\\
  0&0&0&0\\
    0&0&0&0\\
      -2k&0&0&-2k
  \end{array} \right)
\left(\begin{array}{c}
a\\ b\\ c\\ d
\end{array}\right)
= \left(\begin{array}{c}
0\\ 0\\ 0\\ 0
\end{array}\right).$$

Note that $a+d=0$. The rank of ${\bf Q}_{3k}$  is one and therefore there are three linearly independent nullspace vectors of the adjacency matrix given by the generalized kernel eigenvector
$\left(\begin{array}{ccccccccccccc}
a & b & c & -a & -b & -c & \ldots & a & b & c & -a & -b & -c
\end{array}\right)^T$ on 3 parameters.

For $n=3k+1,$ $k\in \mathbb{Z}^+$,   ${\bf Ax}={\bf 0}$  yields
$${\bf Q}_{3k+1}\left(\begin{array}{c}
a\\ b\\ c\\ d
\end{array}\right)
=\left( \begin{array}{cccc}
  -2k-1&0&1&-2k\\
  -1&-1&1&1\\
    -1&-2&-1&0\\
      -2k-1&-1&-1&-2k-1
  \end{array} \right)
\left(\begin{array}{c}
a\\ b\\ c\\ d
\end{array}\right)
= \left(\begin{array}{c}
0\\ 0\\ 0\\ 0
\end{array}\right).$$

For $n=3k+2,$   ${\bf Ax}={\bf 0}$  yields
$${\bf Q}_{3k+2}\left(\begin{array}{c}
a\\ b\\ c\\ d
\end{array}\right)
=\left( \begin{array}{cccc}
  -2k-2&-1&0&-2k-1\\
  -1&-2&-1&0\\
    0&-1&-2&-1\\
      -2k-1&0&-1&-2k-2
  \end{array} \right)
\left(\begin{array}{c}
a\\ b\\ c\\ d
\end{array}\right)
= \left(\begin{array}{c}
0\\ 0\\ 0\\ 0
\end{array}\right).$$

Row reduction of either of the matrices  ${\bf Q}_{3k+i}$  for $i \in \{1,2\}$ gives  that $a+d=0, b-d=0$  and $c+d=0$. In either case, the rank of ${\bf Q}_{3k+i}$  is three and therefore the dimension of the  nullspace of ${\bf A}$ is one. The nullspace vector is given by
$\begin{pmatrix}
a & -a & a & -a & a &  \ldots & a & -a
\end{pmatrix}^T$.

\begin{figure}[h]
\begin{center}
\begin{tikzpicture}[scale=2.2]
\tikzstyle{every path}=[draw, thick]
\tikzstyle{every node}=[draw, circle, fill=red, inner sep=1.5pt]
\node (v_0) at (0.00000, 1.00000) {};
\node (v_1) at (0.86603, -0.50000) {};
\node (v_2) at (0.00000, -1.00000) {};
\node (v_3) at (-0.86603, 0.50000) {};
\node (v_4) at (-0.86603, -0.50000) {};
\node (v_5) at (0.50000, 0.86603) {};
\node (v_6) at (1.00000, 0.00000) {};
\node (v_7) at (0.86603, 0.50000) {};
\node (v_8) at (-0.50000, -0.86603) {};
\node (v_9) at (0.50000, -0.86603) {};
\node (v_10) at (-1.00000, 0.00000) {};
\node (v_11) at (-0.50000, 0.86603) {};
\path (v_0) -- (v_5);
\path (v_0) -- (v_6);
\path (v_0) -- (v_9);
\path (v_0) -- (v_11);
\path (v_1) -- (v_6);
\path (v_1) -- (v_7);
\path (v_1) -- (v_8);
\path (v_1) -- (v_9);
\path (v_2) -- (v_6);
\path (v_2) -- (v_8);
\path (v_2) -- (v_9);
\path (v_2) -- (v_10);
\path (v_3) -- (v_7);
\path (v_3) -- (v_8);
\path (v_3) -- (v_10);
\path (v_3) -- (v_11);
\path (v_4) -- (v_8);
\path (v_4) -- (v_9);
\path (v_4) -- (v_10);
\path (v_4) -- (v_11);
\path (v_5) -- (v_7);
\path (v_5) -- (v_10);
\path (v_5) -- (v_11);
\path (v_6) -- (v_7);
\end{tikzpicture}
\end{center}
\caption{One of the 269 quartic nut graphs of order 12. \label{Fig-Quartic12}
}
\end{figure}
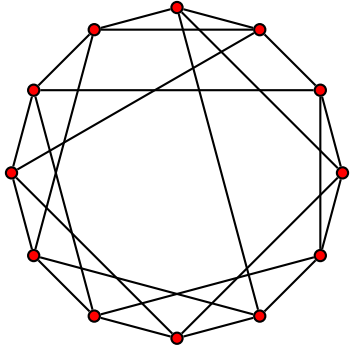

We shall now prove the second main theorem stated in the introduction, namely Theorem \ref{QuarticNut}.

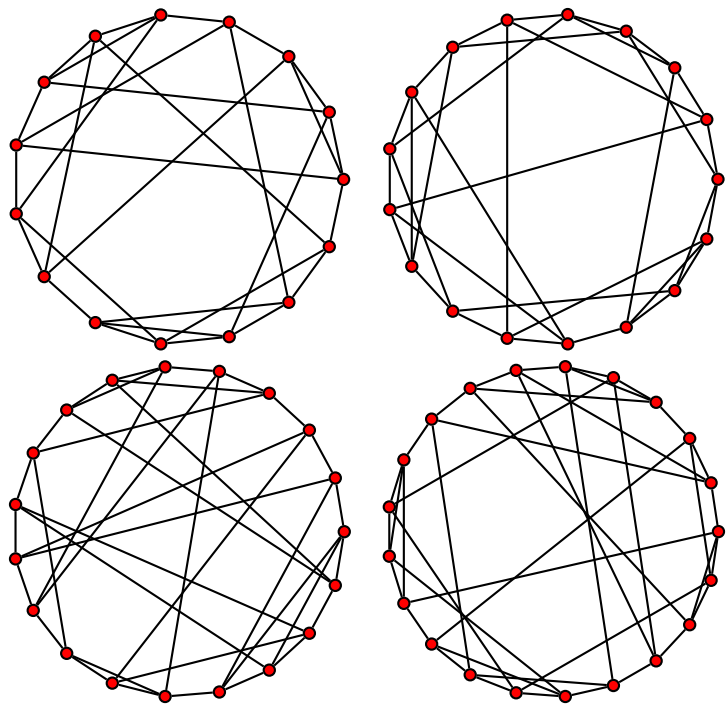
\begin{figure}[h]
\begin{center}
\begin{tabular}{cc}
\begin{tikzpicture}[scale=2.2]
\tikzstyle{every path}=[draw, thick]
\tikzstyle{every node}=[draw, circle, fill=red, inner sep=1.5pt]
\node (v_0) at (0.91355, 0.40674) {};
\node (v_1) at (-0.97815, 0.20791) {};
\node (v_2) at (-0.10453, 0.99452) {};
\node (v_3) at (-0.80902, -0.58779) {};
\node (v_4) at (0.66913, -0.74314) {};
\node (v_5) at (-0.10453, -0.99452) {};
\node (v_6) at (1.00000, 0.00000) {};
\node (v_7) at (-0.50000, 0.86603) {};
\node (v_8) at (0.30902, -0.95106) {};
\node (v_9) at (0.30902, 0.95106) {};
\node (v_10) at (0.66913, 0.74314) {};
\node (v_11) at (-0.80902, 0.58779) {};
\node (v_12) at (-0.97815, -0.20791) {};
\node (v_13) at (0.91355, -0.40674) {};
\node (v_14) at (-0.50000, -0.86603) {};
\path (v_0) -- (v_6);
\path (v_0) -- (v_8);
\path (v_0) -- (v_10);
\path (v_0) -- (v_11);
\path (v_1) -- (v_6);
\path (v_1) -- (v_9);
\path (v_1) -- (v_11);
\path (v_1) -- (v_12);
\path (v_2) -- (v_7);
\path (v_2) -- (v_9);
\path (v_2) -- (v_11);
\path (v_2) -- (v_12);
\path (v_3) -- (v_7);
\path (v_3) -- (v_10);
\path (v_3) -- (v_12);
\path (v_3) -- (v_14);
\path (v_4) -- (v_8);
\path (v_4) -- (v_9);
\path (v_4) -- (v_13);
\path (v_4) -- (v_14);
\path (v_5) -- (v_8);
\path (v_5) -- (v_12);
\path (v_5) -- (v_13);
\path (v_5) -- (v_14);
\path (v_6) -- (v_10);
\path (v_6) -- (v_13);
\path (v_7) -- (v_11);
\path (v_7) -- (v_13);
\path (v_8) -- (v_14);
\path (v_9) -- (v_10);
\end{tikzpicture}

 &

\begin{tikzpicture}[scale=2.2]
\tikzstyle{every path}=[draw, thick]
\tikzstyle{every node}=[draw, circle, fill=red, inner sep=1.5pt]
\node (v_0) at (0.93247, -0.36124) {};
\node (v_1) at (-0.85022, 0.52643) {};
\node (v_2) at (-0.27366, 0.96183) {};
\node (v_3) at (-0.98297, -0.18375) {};
\node (v_4) at (0.73901, 0.67370) {};
\node (v_5) at (-0.60263, -0.79802) {};
\node (v_6) at (1.00000, 0.00000) {};
\node (v_7) at (-0.60263, 0.79802) {};
\node (v_8) at (0.09227, -0.99573) {};
\node (v_9) at (0.09227, 0.99573) {};
\node (v_10) at (0.44574, -0.89516) {};
\node (v_11) at (-0.98297, 0.18375) {};
\node (v_12) at (-0.85022, -0.52643) {};
\node (v_13) at (0.44574, 0.89516) {};
\node (v_14) at (0.93247, 0.36124) {};
\node (v_15) at (-0.27366, -0.96183) {};
\node (v_16) at (0.73901, -0.67370) {};
\path (v_0) -- (v_6);
\path (v_0) -- (v_10);
\path (v_0) -- (v_15);
\path (v_0) -- (v_16);
\path (v_1) -- (v_7);
\path (v_1) -- (v_8);
\path (v_1) -- (v_11);
\path (v_1) -- (v_12);
\path (v_2) -- (v_7);
\path (v_2) -- (v_9);
\path (v_2) -- (v_14);
\path (v_2) -- (v_15);
\path (v_3) -- (v_8);
\path (v_3) -- (v_11);
\path (v_3) -- (v_12);
\path (v_3) -- (v_14);
\path (v_4) -- (v_9);
\path (v_4) -- (v_10);
\path (v_4) -- (v_13);
\path (v_4) -- (v_14);
\path (v_5) -- (v_11);
\path (v_5) -- (v_12);
\path (v_5) -- (v_15);
\path (v_5) -- (v_16);
\path (v_6) -- (v_13);
\path (v_6) -- (v_14);
\path (v_6) -- (v_16);
\path (v_7) -- (v_12);
\path (v_7) -- (v_13);
\path (v_8) -- (v_10);
\path (v_8) -- (v_15);
\path (v_9) -- (v_11);
\path (v_9) -- (v_13);
\path (v_10) -- (v_16);
\end{tikzpicture}
 \\

\begin{tikzpicture}[scale=2.2]
\tikzstyle{every path}=[draw, thick]
\tikzstyle{every node}=[draw, circle, fill=red, inner sep=1.5pt]
\node (v_0) at (1. , 0.) {};
\node (v_7) at (0.945817 , 0.324699) {};
\node (v_1) at (0.789141 , 0.614213) {};
\node (v_16) at (0.546948 , 0.837166) {};
\node (v_9) at (0.245485 , 0.9694) {};
\node (v_18) at (-0.0825793 , 0.996584) {};
\node (v_6) at (-0.401695 , 0.915773) {};
\node (v_10) at (-0.677282 , 0.735724) {};
\node (v_5) at (-0.879474 , 0.475947) {};
\node (v_14) at (-0.986361 , 0.164595) {};
\node (v_11) at (-0.986361 , -0.164595) {};
\node (v_3) at (-0.879474 , -0.475947) {};
\node (v_15) at (-0.677282 , -0.735724) {};
\node (v_12) at (-0.401695 , -0.915773) {};
\node (v_4) at (-0.0825793 , -0.996584) {};
\node (v_13) at (0.245485 , -0.9694) {};
\node (v_8) at (0.546948 , -0.837166) {};
\node (v_2) at (0.789141 , -0.614213) {};
\node (v_17) at (0.945817 , -0.324699) {};
\path (v_0) -- (v_7);
\path (v_0) -- (v_8);
\path (v_0) -- (v_13);
\path (v_0) -- (v_17);
\path (v_1) -- (v_7);
\path (v_1) -- (v_11);
\path (v_1) -- (v_12);
\path (v_1) -- (v_16);
\path (v_2) -- (v_8);
\path (v_2) -- (v_12);
\path (v_2) -- (v_14);
\path (v_2) -- (v_17);
\path (v_3) -- (v_9);
\path (v_3) -- (v_11);
\path (v_3) -- (v_15);
\path (v_3) -- (v_18);
\path (v_4) -- (v_9);
\path (v_4) -- (v_12);
\path (v_4) -- (v_13);
\path (v_4) -- (v_15);
\path (v_5) -- (v_10);
\path (v_5) -- (v_14);
\path (v_5) -- (v_15);
\path (v_5) -- (v_16);
\path (v_6) -- (v_10);
\path (v_6) -- (v_16);
\path (v_6) -- (v_17);
\path (v_6) -- (v_18);
\path (v_7) -- (v_11);
\path (v_7) -- (v_13);
\path (v_8) -- (v_13);
\path (v_8) -- (v_14);
\path (v_9) -- (v_16);
\path (v_9) -- (v_18);
\path (v_10) -- (v_17);
\path (v_10) -- (v_18);
\path (v_11) -- (v_14);
\path (v_12) -- (v_15);
\end{tikzpicture}

&

\begin{tikzpicture}[scale=2.2]
\tikzstyle{every path}=[draw, thick]
\tikzstyle{every node}=[draw, circle, fill=red, inner sep=1.5pt]
\node (v_13) at (1. , 0.) {};
\node (v_1) at (0.955573 , 0.294755) {};
\node (v_9) at (0.826239 , 0.56332) {};
\node (v_17) at (0.62349 , 0.781831) {};
\node (v_12) at (0.365341 , 0.930874) {};
\node (v_6) at (0.0747301 , 0.997204) {};
\node (v_20) at (-0.222521 , 0.974928) {};
\node (v_5) at (-0.5 , 0.866025) {};
\node (v_18) at (-0.733052 , 0.680173) {};
\node (v_10) at (-0.900969 , 0.433884) {};
\node (v_2) at (-0.988831 , 0.149042) {};
\node (v_15) at (-0.988831 , -0.149042) {};
\node (v_3) at (-0.900969 , -0.433884) {};
\node (v_16) at (-0.733052 , -0.680173) {};
\node (v_8) at (-0.5 , -0.866025) {};
\node (v_14) at (-0.222521 , -0.974928) {};
\node (v_7) at (0.0747301 , -0.997204) {};
\node (v_19) at (0.365341 , -0.930874) {};
\node (v_4) at (0.62349 , -0.781831) {};
\node (v_11) at (0.826239 , -0.56332) {};
\node (v_0) at (0.955573 , -0.294755) {};
\path (v_0) -- (v_9);
\path (v_0) -- (v_11);
\path (v_0) -- (v_13);
\path (v_0) -- (v_14);
\path (v_1) -- (v_9);
\path (v_1) -- (v_13);
\path (v_1) -- (v_18);
\path (v_1) -- (v_20);
\path (v_2) -- (v_10);
\path (v_2) -- (v_12);
\path (v_2) -- (v_14);
\path (v_2) -- (v_15);
\path (v_3) -- (v_10);
\path (v_3) -- (v_13);
\path (v_3) -- (v_15);
\path (v_3) -- (v_16);
\path (v_4) -- (v_11);
\path (v_4) -- (v_12);
\path (v_4) -- (v_19);
\path (v_4) -- (v_20);
\path (v_5) -- (v_11);
\path (v_5) -- (v_17);
\path (v_5) -- (v_18);
\path (v_5) -- (v_20);
\path (v_6) -- (v_12);
\path (v_6) -- (v_17);
\path (v_6) -- (v_19);
\path (v_6) -- (v_20);
\path (v_7) -- (v_14);
\path (v_7) -- (v_15);
\path (v_7) -- (v_16);
\path (v_7) -- (v_19);
\path (v_8) -- (v_14);
\path (v_8) -- (v_16);
\path (v_8) -- (v_18);
\path (v_8) -- (v_19);
\path (v_9) -- (v_16);
\path (v_9) -- (v_17);
\path (v_10) -- (v_15);
\path (v_10) -- (v_18);
\path (v_11) -- (v_13);
\path (v_12) -- (v_17);
\end{tikzpicture}
\end{tabular}
\end{center}
\caption{Quartic nut graphs of order 15, 17, 19 and 21.} \label{Fig-QuarticODD}
\end{figure}

\begin{proof} (Proof of  Theorem \ref{QuarticNut})
Starting from the antiprism graphs $A_4, A_5$ and $A_7$ implies that there exist quartic nut graphs on 8, 10 and 14 vertices, respectively. Computer search has revealed 269 quartic nut graphs on 12 vertices, one of which is shown in  Figure \ref{Fig-Quartic12}. Using the graphs of order $8,10,12$ and $14$ as seeds for repeated Fowler constructions takes care of all even orders at least eight.

For the odd case, an exhaustive computer search has shown that no quartic nut graphs of odd order at most 13 exist and there is only one quartic nut graph of order 15 (shown in Figure \ref{Fig-QuarticODD}). Quartic nut graphs of order $17, 19$ and $21$ are also shown in Figure \ref{Fig-QuarticODD}. Using these orders as seeds for repeated Fowler constructions proves that quartic nut graphs of odd order at least 13 exist.
\end{proof}


\section{An open problem}

In this paper we have determined the values of $n$ for which there exist $\rho$-regular nut graphs of order $n$ for $\rho=3,4$. These results lead to a very natural question, namely:

\begin{problem}
For each degree $\rho$, determine the set $N(\rho)$ such that there exists a $\rho$-regular nut graph of order $n$ if and only if $n \in N(\rho)$.
\end{problem}


\section*{Acknowledgements}

The authors would like to thank Nino Ba\v{s}i\'{c} for conducting the computer searches, Xandru Mifsud for programs yielding the kernel eigenvectors of $A_n$, and Jan Goedgebeur for drawing our attention to an inaccuracy in the number of cubic nut graphs. The second author would like to acknowledge support of the ARRS (Slovenia) grants P1-0294, J1-9187 and N1-0032, and thank the Department of Mathematics of the University of Malta for the warm welcome during his visit that made this research possible.



\end{document}